\documentclass[11pt]{article}
\usepackage{amsthm, amsmath, amssymb, amsfonts, url, booktabs, tikz, setspace, fancyhdr, enumerate}
\usepackage[margin = 1in]{geometry}

\usepackage[hidelinks]{hyperref}
\usepackage{cleveref}

\usepackage[title]{appendix}

\newtheorem{theorem}{Theorem}[section]

\newtheorem{lemma}[theorem]{Lemma}
\newtheorem{corollary}[theorem]{Corollary}
\newtheorem{conjecture}[theorem]{Conjecture}
\newtheorem{claim}[theorem]{Claim}

\theoremstyle{definition}

\theoremstyle{remark}

\newcommand{\U}{\mathcal{U}}

\tikzstyle{p}+=[fill=black, circle, minimum width = 1pt, inner sep =
1pt]
\tikzstyle{w}+=[fill=white, draw, circle, minimum width = 1pt, inner sep =
1.5pt]

\def\abhi#1{}
\let\abhi=\abhiOpt

\begin{document}

\title{A proof of the Elliott--R\"{o}dl conjecture on hypertrees \\ in Steiner triple systems}

\author{
Seonghyuk Im\thanks{Department of Mathematical Sciences, KAIST, South Korea and Extremal Combinatorics and Probability Group (ECOPRO), Institute for Basic Science (IBS). Email:{\tt seonghyuk@kaist.ac.kr}}
\and
Jaehoon Kim\thanks{Department of Mathematical Sciences, KAIST, South Korea. Email: {\tt jaehoon.kim@kaist.ac.kr}.}
\and 
Joonkyung Lee\thanks{Department of Mathematics, Hanyang University, Seoul, South Korea and Extremal Combinatorics and Probability Group (ECOPRO), Institute for Basic Sciences (IBS).
E-mail: {\tt
joonkyunglee@hanyang.ac.kr}.}
\and
Abhishek Methuku\thanks{School of Mathematics, University of Birmingham,
%Edgbaston, Birmingham, B15 2TT, 
United Kingdom.
E-mail: {\tt
abhishekmethuku@gmail.com}.}
}

\date{}

\maketitle

\begin{abstract}
Hypertrees are linear hypergraphs where every two vertices are connected by a unique path.
Elliott and R\"{o}dl conjectured that for any given $\mu>0$, there exists $n_0$ such that the following holds. 
Every $n$-vertex Steiner triple system contains all hypertrees with at most $(1-\mu)n$ vertices whenever $n\geq n_0$. We prove this conjecture.
\end{abstract}

\section{Introduction}

A key question in extremal graph theory is to find (almost-)spanning subgraphs $H$ in a host graph~$G$ chosen from a certain class of graphs. 
Perhaps the simplest case along these lines may be to find (almost-)perfect matchings or (almost-)spanning trees, but even these cases have led us to a profound and intricate theory. 

When the host graph $G$ is dense, there have been extensive studies relating the existence of such (almost-)spanning subgraphs and the minimum degree of $G$.
One of the earliest examples may be Dirac's theorem, which states that the minimum degree condition $\delta(G)\geq n/2$ implies the existence of Hamiltonian cycles and paths. 
Another well-known example is that
every $n$-vertex graph with the minimum degree $\delta(G)$ contains a matching with at least $\min\{\delta(G),\lfloor n/2 \rfloor\}$ edges. In particular, $\delta(G)\geq n/2$ implies that there exists a perfect matching in a graph, provided that $n$ is even. 
This fact was subsequently generalized by Brandt~\cite{Brandt1994}, who proved that a graph $G$ contains every forest $F$ with at most $\delta(G)$ edges and at most $n$ vertices. 
Note that this also includes another well-known fact that every graph $G$ contains all trees with at most $\delta(G)$ edges. 
There have been more results along these lines, e.g., a theorem by Koml\'os, S\'ark\"{o}zy, and Szemer\'edi \cite{Komlos2001} which shows the existence of spanning trees under certain degree conditions or the bandwidth theorem by B\"ottcher, Schacht and Taraz~\cite{Boettcher2009}.

In contrast, finding (almost)-spanning structures in a sparse host graph $G$ often turns out to be impossible in general, 
which forces one to consider more restricted classes of graphs. For example, a classical theorem of P\'osa \cite{Posa1976}, also obtained by Koml\'{o}s and Szemer\'{e}di \cite{KS1983}, states that ``typical" graphs with at least $(1+o(1))n\log n$ edges contain a Hamilton cycle.
Recently, Montgomery \cite{Montgomery2019} proved that typical graphs with $\Omega(n\log n)$ edges contain all spanning trees with bounded maximum degree. Other classes of graphs have also been considered, e.g., a ``resilience" version of these results \cite{Balogh2011, Lee-sudakov2012} 
or variants for randomly perturbed graphs \cite{Bohman-Frieze-Martin03, Boettcher2019, Joos-Kim20, Krivelevich-Kwan-Sudakov17}. 

Although it is known that some large minimum codegree conditions on hypergraphs ensure existence of perfect matchings, spanning trees or cycles \cite{Kuhn-Mycroft-Osthus10, Pavez-signe-Sanhueza-Matamala-Stein21, rodl-Rucinski-Szemeredi06, rodl-Rucinski-Szemeredi08, rodl-Rucinski-Szmeredi09}, 
the problem of finding spanning (or almost spanning) subhypergraphs in hypergraphs is fundamentally different from graphs. 
Our main goal is to explore this relatively less discovered area by proving the existence of an almost-spanning ``hypertree" in a well-known class of ``sparse" 3-uniform hypergraphs, the so-called \emph{Steiner triple systems}.
To further discuss our result, we first clarify what hypertrees are.

In what follows, we restrict our attention only to $3$-uniform hypergraphs (or $3$-graphs briefly). A $3$-graph is \textit{linear} if every pair of distinct edges has at most one vertex in common. A \textit{hypertree} is a connected, linear $3$-graph in which any two vertices are connected by a unique path (see \Cref{fig:bare_path}). Equivalently, a hypertree can be obtained by recursively adding edges such that each new edge intersects the current set of vertices in exactly one vertex.  A \emph{matching} in a $3$-graph $H$ is a collection of pairwise disjoint edges of $H$ and a \textit{perfect matching} is a matching that covers all the vertices in $H$.

\begin{figure}
    \centering
    \includegraphics[width=\textwidth]{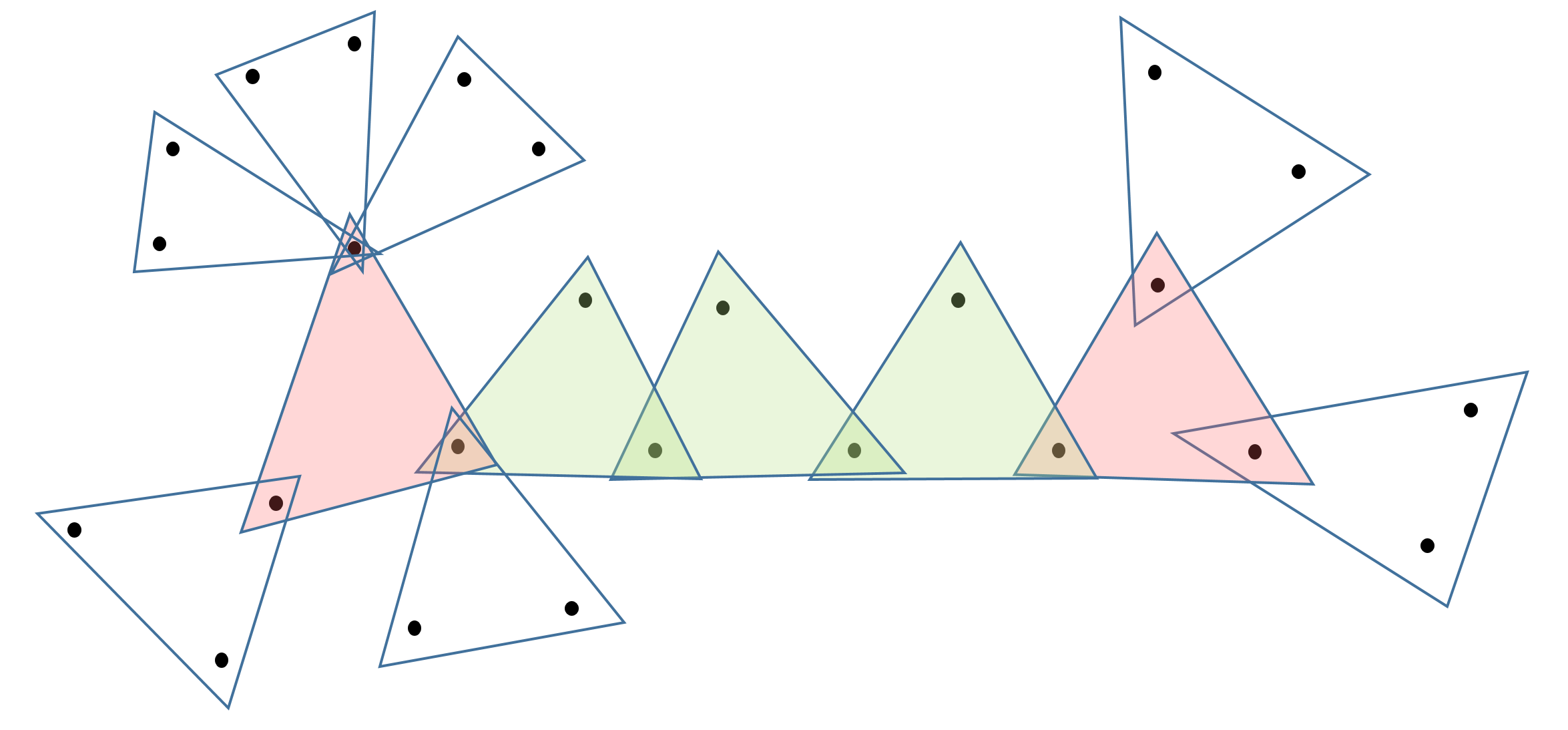}
    \caption{An example of a hypertree}
    \label{fig:bare_path}
\end{figure}

As mentioned before, (spanning) hypertrees in hypergraphs behave in a fundamentally different way from the graph case.
For example, any connected graph contains at least one spanning tree; however, for $3$-graphs, this may not be the case. First, a $3$-graph with an even number of vertices has no spanning tree at all, as a hypertree always has an odd number of vertices.
Even if we assume that the number of vertices is odd to set aside the parity issue, connectivity alone is still not sufficient to find a large hypertree. For example, a hypergraph with $2k+1$ edges sharing the same two vertices forms a simplest $3$-graph with no hypertree having more than one edge.
For another example illustrating the difference, a simple greedy algorithm finds all trees with $\delta(G)$ edges in a graph $G$ with the minimum degree $\delta(G)$, but the same algorithm for $3$-graphs only yields all hypertrees with at most $\frac{1}{2}(\delta(G)+1)$ edges. 

On the other hand, Goodall and de Mier \cite{goodall-deMier2011} proved that every $3$-graph with odd number of vertices where every pair belongs to at least one edge contains at least one spanning hypertree. The extremal cases in their theorem are the $3$-graphs where every pair of vertices belongs to exactly one hyperedge, i.e., Steiner triple systems.

teiner triple systems can be seen as an analogue of complete graphs (where all possible types of spanning trees exist) for linear $3$-graphs. However, even though the Goodall--de Mier theorem proves that every Steiner triple system contains a spanning hypertree (in fact, they proved that an $n$-vertex Steiner triple system contains at least $\Omega((n/6)^{n/12})$ spanning trees), it is far from the truth that every Steiner triple system contains all types of spanning hypertrees. 
For example, it is known \cite{bryant2015steiner} that for infinitely many odd $n$, there exist $n$-vertex Steiner triple systems with no perfect matching. 
Thus, any hypertree containing a perfect matching (there are super-exponentially many such hypertrees) cannot be found in those Steiner triple systems. This motivates the following natural question: what is the largest number $t$ such that any $n$-vertex Steiner triple system contains all hypertrees on $t$ vertices? 

This extremal question is hard even for matchings, which have a lot simpler structure than arbitrary hypertrees. 
A famous forty-year-old conjecture of Brouwer~\cite{brouwer1981size} states that every $n$-vertex Steiner triple system contains a matching covering $n-4$ vertices. 
This conjecture remains open and the best bound known so far is the recent progress by Keevash, Pokrovskiy, Sudakov and Yepremyan~\cite{keevash2022}, proving that any $n$-vertex Steinter triple system contains a matching covering at least $n- O(\frac{\log n}{\log\log n})$ vertices. 
These results on matchings already allude to the fact that determining the exact value of $t$ may be out of reach at the moment. 

In 2019, Elliot and R\"odl asked an ``asymptotic" question, which appears to be the very first step towards determining the exact value of $t$.
\begin{conjecture}[Elliott and  R\"{o}dl \cite{Elliott2019}]
\label{Elliott_Rodl}
Given $\mu\geq 0$, there exists $n_0=n_0(\mu)$ such that if $n \geq n_0$, $T$ is any
hypertree on $n$ vertices, and $S$ is any Steiner triple system on $m\geq n(1+\mu)$ vertices, then
T is a subhypergraph of S.
\end{conjecture}
The conjecture implies that, although there are super-exponentially many spanning hypertrees which cannot be found in some Steiner triple systems, a completely different behavior should exist for slightly smaller hypertrees.
 Elliott and  R\"{o}dl gave a positive evidence for  the conjecture by proving it for special types of hypertrees called \textit{subdivision trees}.
Later, Arman, R\"{o}dl, and Sales \cite{arman2021colourful,arman2022} proved Conjecture~\ref{Elliott_Rodl} for two different classes of hypertrees called \textit{$d$-ary hypertrees} and \textit{turkeys}. 
Our main result is to prove Conjecture~\ref{Elliott_Rodl} completely (i.e for all hypertrees).

\begin{theorem}\label{main_thm}
 For every $\varepsilon>0$, there exists $n_0=n_0(\varepsilon)$ such that every $n$-vertex Steiner triple system $G$ with $n \geq n_0$ contains every hypertree with at most $(1-\varepsilon) n$ vertices.
\end{theorem}

\section{An outline of the proof}

Our proof of \Cref{main_thm} is partly motivated by the rainbow tree embedding theorem of Montgomery, Pokrovskiy, and Sudakov~\cite{Montgomery2020}, which states that there exists a ``rainbow" copy of every tree $T$ with at most $(1-o(1))n$ vertices in a properly edge-colored copy of $K_n$.
As we consider $3$-uniform linear hypergraphs instead of graphs, there are technical challenges to overcome, which will be discussed in due course.

Let $T$ be a hypertree with $(1-o(1))n$ vertices and let $G$ be an $n$-vertex Steiner triple system. We first decompose $T$ into $T_0 \subseteq T_1 \subseteq \cdots \subseteq T_\ell = T$ such that 
\begin{enumerate}[(1)]
    \item \label{it:o(1)} $|T_0| = o(n)$,
    \item $\ell=O(\text{polylog}(n))$,
    \item\label{large_star} $T_1$ is obtained by adding ``large" stars to $T_0$, and
    \item\label{path_or_matching} each $T_i$, $i>1$, is obtained by adding either $o(n)$ many paths of length $3$ whose endpoints are in $T_{i-1}$ or a matching where each edge contains exactly one vertex in $T_{i-1}$.\label{tree_spliting_4}
\end{enumerate} 
We will embed $T$ into $G$ by first embedding $T_0$ (in Step 0), and then extend $T_{i-1}$ to $T_{i}$ (in Step $i$) for each $i \ge 1$. We will treat the initial steps, embedding $T_0$ and extending $T_0$ to $T_1$, with extra care and execute the remaining steps inductively.
More precisely, we first embed $T_0$ by a deterministic greedy algorithm and collect a set $\mathcal{S}$ of ``large" vertex-disjoint stars  centered at vertices in the copy of $T_0$ (where a star centered at a vertex $v$ is a set of edges whose pairwise intersection is $\{v\}$). Here the stars we find are in fact almost spanning, i.e. ``larger" than the large stars described in \eqref{large_star} due to a technical reason.
We then partition all the $(1-o(1))n$ vertices of $G$ that are not in the copy of $T_0$ into $R\cup X_1\cup\cdots\cup X_\ell$ in a randomized way while also making sure that the two leaves in every edge of every star in $\mathcal{S}$ lie in the same part of the partition.
Given the partition, ``most" vertices in $T_i\setminus T_{i-1}$ are embedded into $X_i$ and the remaining vertices are embedded into the ``reservoir" $R$ at each step.

\medskip

Let us take a closer look at each of the steps. First, let us consider Steps 0, 1.
The decomposition $T_0 \subseteq T_1 \subseteq \cdots \subseteq T_\ell$ of $T$ closely follows the ideas of~\cite{Montgomery2020}.
Finding $\mathcal{S}$ is not hard either; it follows from adapting the switching technique in Section 8 of~\cite{Montgomery2020}, originally introduced by Woolbright \cite{Woolbright1978} and Brouwer, de Vries, and Wieringa \cite{MR480083}.
A technical problem occurs when we build the vertex set $X_1$.
In~\cite{Montgomery2020}, $X_1$ was simply taken by choosing each vertex independently at random, but in our $3$-graph setting, 
a straightforward application of this approach no longer captures the structures of stars that we need, as the ``leaves" of a star in a $3$-graph are pairs of vertices rather than singletons; for instance, too few edges of a star will be selected if $X_1$ is taken by choosing vertices uniformly at random. To resolve this issue, we pair the leaf vertices of each star and we either select both vertices in each pair, or we select neither of them while choosing the random subset $X_1$.

Now let us consider Step $i$ for $i >1$. Let $\{X_{i, 1}, X_{i, 2}\}$ be a random partition of $X_i$. For each $i>1$, Step $i$ uses the following properties of the partition $R\cup (X_{1,1}\cup X_{1,2})\cup (X_{2,1}\cup X_{2,2})\cup\cdots\cup(X_{\ell,1}\cup X_{\ell,2})$, which can be obtained by using standard concentration inequalities. For each $i >1$ and $j\in [2]$, 
\begin{enumerate}[(i)]
    \item every vertex has ``many" neighbors in $X_{i,j}$ as well as in $R$, and \label{random_1}
    \item for every pair of large enough disjoint sets $A, B \subseteq V(G) \setminus X_{i,j}$, the number of edges of the form $\{a,b,c\}$ with  $a \in A$, $b \in B$ and $c \in X_{i,j}$ is close to their expectation. \label{random_2}
\end{enumerate}
To perform \eqref{path_or_matching}, one should find $o(n)$ paths of length three or a matching to extend $T_{i-1}$ to $T_{i}$.
In the former case, it is not hard to find $o(n)$ paths of length three by using \eqref{random_1} and \eqref{random_2} above, and this follows the ideas of~\cite{Montgomery2020} closely. 
However, we use a completely different approach for finding a matching to extend $T_{i-1}$ to $T_i$. 
 Using \eqref{random_2} and Pippenger's hypergraph matching theorem, we embed most of the edges of the matching using vertices in $X_i = X_{i, 1} \cup X_{i,2}$.
Then we embed the remaining edges of the matching using vertices in $R$ by making use of ~\eqref{random_1}.

\medskip

\textbf{Organization.}
We introduce
preliminary results in Section~\ref{sec_preliminary}. 
We show \eqref{random_1} and \eqref{random_2} are true for our random process and use it to embed matchings and bare paths in Section~\ref{sec_matching}.
In Section~\ref{sec_stars}, we show that we can embed large vertex-disjoint stars into a Steiner triple system with prescribed centers.
Finally, we put everything together to iteratively find an embedding of $T$ in Section~\ref{sec_combine}.
The proof of hypertree splitting lemma is illustrated in Appendix~\ref{sec_tree_partition}.

\section{Preliminaries and notation}\label{sec_preliminary}

As outlined in the previous section, we need to decompose a hypertree $T$ into $T_0\subseteq T_1\subseteq\cdots\subseteq T_\ell=T$, following the approach taken in~\cite{Montgomery2020}.
To this end, we introduce some auxiliary definitions.

First of all, let us clarify our definition of paths. 
A \emph{(Berge-)path of length $\ell$} in a hypergraph is a sequence $v_1,e_1,v_2,\dots,e_\ell,v_{\ell+1}$ of distinct vertices and edges such that $v_i,v_{i+1}\in e_i$ for all $i\in [\ell]$. In particular, a path in a linear $3$-graph $G$ is a subgraph of $G$ on $2\ell+1$ vertices $\{v_0,v_1,\cdots,v_\ell\}\cup\{u_1,\cdots,u_{\ell}\}$ such that each $\{v_i,u_{i+1},v_{i+1}\}$ is an edge for $i=0,1,\cdots,\ell-1$.
Each of the two pairs $\{v_0,u_1\}$ and $\{u_\ell,v_\ell\}$ in the first and the last edges of $P$ is called an \emph{end pair} of $P$.

A hypergraph $T$ is a hypertree if and only if there is a unique path between any pair of distinct vertices. It is straightforward to see that a hypertree is a linear hypergraph.
A \emph{$u$--$v$ path} $P$ (or a path $P$ between $u$ and $v$) means a path $P$ paired with the specified end vertices $u$ and~$v$, each of which is chosen from each end pair, respectively.
The vertices other than $u$ and $v$ in a $u$--$v$ path $P$ are called the \emph{internal} vertices of $P$.

In a hypertree $T$, a \emph{bare path} $P$ is a subhypergraph of $T$ such that it is a $u$--$v$ path of length $\ell \geq 2$ where no edges in $T\setminus E(P)$ are incident to the internal vertices of $P$.
For example, the green edges in \Cref{fig:bare_path} form a bare path but green edges plus one of the red edges does not.

A \emph{leaf} of a hypertree $T$ is a vertex $v\in V(T)$ of degree one such that the edge $e$ containing $v$ has another vertex $u$ of degree one. That is, removing $u$ and $v$ from $T$ produces a subhypertree of~$T$.
The edges that contain a leaf are called \emph{leaf edges} of $T$.
In particular, the number of leaves of $T$ is always even unless $T$ is a single edge.

A \emph{star} of size $D$ is a 3-graph $S$ on $2D+1$ vertices $\{v\}\cup\{u_1,\cdots,u_D\}\cup\{w_1,\cdots,w_D\}$ such that each triple $\{v,u_i,w_i\}$ is an edge for $i \in [D]$.
The vertex $v$ of degree $D$ is called the \emph{center} of $S$.
A \emph{matching}~$M$ is a collection of pairwise disjoint edges. 
If a matching $M$ consists of leaf edges of a hypertree~$T$, the set of leaves of $T$ in $M$ is called the \emph{leaf set} of $M$. 
We simply say that a vertex subset $X$ of $T$ is a  \emph{matching leaf set} of $T$ if $X$ is the leaf set of a matching $M$ in $T$.

\medskip

Using these terminologies, the following lemma formalizes \eqref{it:o(1)}--\eqref{path_or_matching} in the previous section by adapting the tree splitting lemma in Section~4 of~\cite{Montgomery2020} to our linear 3-graph setting.
\begin{lemma}\label{lem_tree_split}
    Let $D, n \geq 2$ be integers and let $0<\mu<1$.
    For any hypertree $T$ with at most $n$ edges and $2n+1$ vertices, there exist integers $\ell \leq 10^5 D \mu^{-2}$ and $s \in [\ell]$ and a sequence of subgraphs $T_0 \subseteq T_1 \subseteq \cdots T_\ell = T$ such that the following holds:
    \begin{enumerate}[(i)]
       \item\label{it:T0} $T_0$ has at most $\mu n$ edges and at most $3\mu n$ vertices.
       \item\label{it:T1} $T_1$ is obtained by adding stars of size at least $D$ to $T_0$; that is, take pairwise vertex-disjoint stars of size at least $D$ and identify their centers with vertices in $T_0$.
       \item\label{it:matching} For $i \notin\{ 0, s\}$, 
       $T_{i+1}$ is obtained by adding a matching to $T_i$ such that $V(T_{i+1})\setminus V(T_i)$ is a matching leaf set of $T_{i+1}$.
       \item\label{it:path} $T_{s+1}$ is obtained by adding at most $\mu n$ vertex-disjoint bare paths of length $3$ to $T_s$ 
       such that every bare path we add is a $u$--$v$ path $P$ where $u,v\in V(T_s)$, and $V(P) \setminus \{u,v\}$ is disjoint from $V(T_s)$. 
   \end{enumerate}
\end{lemma}

Although this lemma does not seem to follow directly from~\cite{Montgomery2020}, the proof closely resembles theirs. Thus, we will give a brief proof in Appendix~\ref{sec_tree_partition}.

\medskip

We will frequently use standard concentration inequalities, which can be found, for example, in~\cite{MR3524748}.

\begin{lemma}[One-sided Chernoff Bound]\label{lem_oneside_chernoff}
    Let $X_1, \cdots, X_n$ be mutually independent Bernoulli random variables and let $X  = \sum_{i=1}^n X_i$. Then, 
    \begin{align*}
        \mathbb{P}(X \geq (1+\varepsilon)\mathbb{E}[X]) & \leq \exp\left(-\frac{\varepsilon^2}{\varepsilon+2}\mathbb{E}[X]\right) \enspace \text{ for every } \enspace \varepsilon>0, \text{ and }\\
        \mathbb{P}(X \leq (1-\varepsilon)\mathbb{E}[X]) & \leq \exp\left(-\frac{\varepsilon^2}{2}\mathbb{E}[X]\right) \enspace \text{ for every } \enspace \varepsilon \in (0, 1).
    \end{align*}
\end{lemma}
The following corollary will be enough in most of the applications.
\begin{corollary}[The Chernoff Bound]\label{lem_chernoff}
Let $X_1, X_2, \cdots, X_n$ be i.i.d. Bernoulli random variables and let $X= \sum_{i=1}^n X_i$. Then for $\varepsilon \in (0, 1)$, $$\mathbb{P}(|X-\mathbb{E}[X]| \geq \varepsilon \mathbb{E}[X] ) \leq 2\exp\left(-\frac{\varepsilon^2 \mathbb{E}[X]}{3}\right).$$     
\end{corollary}
Given a probability space $\Omega = \prod_{i=1}^n \Omega_i$ and a random variable $X: \Omega \rightarrow \mathbb{R}$, $X$ is \textit{$k$-Lipschitz} if $|X(\omega)-X(\omega')| \leq k$ whenever $\omega$ and $\omega'$ differ in at most one coordinate.
\begin{lemma}[Azuma's inequality]\label{lem_azuma}
If $X: \prod_{i=1}^n \Omega_i \rightarrow \mathbb{R}$ is $k$-Lipschitz, then $$ \mathbb{P}(|X-\mathbb{E}[X]| > t) \leq 2\exp\left(-\frac{t^2}{k^2 n}\right). $$
\end{lemma}

We use the notation $\ll$ for the hierarchy between constants. 
For instance, if we claim that a statement holds under the condition $0<a \ll b, c \ll d$, then it means that there exist non-decreasing functions $f_1, f_2:(0, 1] \rightarrow (0, 1]$ and $g:(0, 1]^2 \rightarrow (0, 1]$ such that the statement holds if $0<b \leq f_1(d)$, $0<c \leq f_2(d)$, and $0<a \leq g(b, c)$. We do not attempt to  describe all these functions explicitly.
We also write $a=(b \pm c)d$ if $(b-c)d \leq a \leq (b+c)d$. 

For a given $3$-graph $G$ and $v \in V(G)$, $\deg(v)$ denotes the number of edges containing $v$ and a vertex $u \in V(G)$ is a \emph{neighbor} of another vertex $v \in V(G)$ if there exists an edge of $G$ containing both $u$ and $v$. Let $\delta(G)=\min_{v\in V(G)} \deg(v)$ be the minimum degree of $G$.
In particular, the number of neighbors of $v$ in a linear 3-graph $G$ is exactly $2 \deg(v)$. 
Let $G[A_1, A_2, A_3]$ denote the subhypergraph of $G$  with the vertex set $A_1 \cup A_2 \cup A_3$ and the edge set $\{\{x,y,z\} \in E(G) : x \in A_1$, $y \in A_2, z \in A_3 \}$. We allow non-empty intersection of $A_1, A_2, A_3$ so an edge in $G[A_1, A_2, A_3]$ may have more than one vertex contained in $A_i$ for some $i \in [3]$. Let $e_G(A_1,A_2,A_3)$ (or simply $e(A_1,A_2,A_3)$) be the number of edges of $G[A_1, A_2, A_3]$.

\medskip

With these terminologies, Pippenger's theorem~\cite{MR993646}, which strengthens a theorem in~\cite{MR829351} (which also appears in~\cite{MR3524748} as a standard application of R\"odl's nibble), can be stated as follows.

\begin{lemma}[Pippenger's theorem]\label{lem_Pippenger}
 Let $\varepsilon, \delta>0$ and let $D, k$ be integers such that $0<\frac{1}{D}, \delta \ll \frac{1}{k}, \varepsilon$. 
 %\abhi{$k \to \frac{1}{k}$ in the hierarchy}
 If an $n$-vertex $k$-uniform hypergraph $H$ satisfies that \begin{enumerate}[(i)]
     \item the degree of $v$, $\deg (v) = (1 \pm \delta) D$ for all $v \in V(H)$ and
     \item for any two distinct vertices $u, v \in V(H)$, the number of edges in $H$ containing $\{u, v\}$ is less than $\delta D$,
 \end{enumerate}
 then $H$ contains a matching of size at least $(1-\varepsilon)\frac{n}{k}$.
\end{lemma}

\section{Matchings and bare paths}\label{sec_matching}

Our goal in this section is to collect lemmas for extending a given hypertree $T_i$ to $T_{i+1}$ when $T_{i+1}$ is obtained by adding a matching or bare paths to $T_i$.
The key tool to handle the former case is Pippenger's theorem and
the latter case follows from standard concentration results for random edge subsets.

Let $G$ be an $n$-vertex Steiner triple system.
A \emph{singleton-pair partition} $\mathcal{U}$ of $V(G)$ is a partition $\{U_1, U_2, \cdots, U_m\}$ of $V(G)$ such that $|U_i|\leq 2$ for each $i\in[m]$.
Throughout this section, $G$ always denotes an $n$-vertex Steiner triple system with a given singleton-pair partition $\U$. 
For such a fixed singleton-pair partition $\U$, the random subset $\mathcal{U}_p$ of $\U$ is then obtained by choosing each $U_i$ independently at random with probability $p$. Let $X=\bigcup_{U\in \mathcal{U}_p} U$, which we denote by $X\sim \mathcal{U}_p$.

\medskip

In order to find a large matching using Pippenger's theorem (Lemma~\ref{lem_Pippenger}), we need to prove that a randomly chosen vertex set $X\sim \U_p$ together with certain sets $A, B$ of already-embedded vertices induces an almost-regular hypergraph (see Lemma~\ref{lem_edges_largeset}). As we cannot specify these sets of already-embedded vertices at the beginning, we want to show that $X\sim \U_p$ together with any pair of disjoint vertex sets $A, B$ of size at least $\Omega(n)$ induces an almost-regular hypergraph. 
However, as there are too many choices of such vertex sets $A$ and $B$, 
naive applications of concentration results do not yield a strong enough result to use the union bound. Hence, we first prove Lemma~\ref{lem_edges_smallset} which shows that $X\sim \U_p$ together with any pair of much smaller sets induces an almost-regular hypergraph, and we use these much smaller sets as building blocks for $A$ and $B$ to prove Lemma~\ref{lem_edges_largeset}. 

\begin{lemma}\label{lem_edges_smallset}
 Let $X\sim \U_p$ for $p \geq 1/(\log n)^{100}$ and let $\varepsilon \geq 1/{(\log n)^{100}}$. Then the following holds with probability at least $1-o(1/n)$:
 For every pair of disjoint sets $A, B \subseteq V(G)$ satisfying 
 \begin{enumerate}[(a)]
     \item \label{6.1.1} $n^{2/5} \leq |A|, |B| \leq n^{1/2}$,
     \item \label{6.1.2} $e(A,B,\{v\})\leq (\log n)^{10}$ for every $v \in V(G)$, and
     \item $X$ is disjoint from $A\cup B$,
 \end{enumerate}
 we have $e(A,B,X)=(1\pm \varepsilon)p|A||B|$.
\end{lemma}
\begin{proof}
    Let $G_X[A, B]$ be the auxiliary bipartite graph on the bipartition $A\cup B$, where $(a,b)\in A\times B$ is an edge if there exists $x\in X$ such that $abx\in E(G)$.  
    As $G$ is a Steiner triple system, every pair $(a,b)\in A\times B$ extends to a unique edge together with a vertex $x$, which belongs to $X$ with probability $p$. Thus, the expected size of $E(G_X[A, B])$ is $p|A||B|$.
    
    One may then apply Azuma's inequality to prove concentration results for $|E(G_X[A, B])|$.
    Indeed, for a fixed pair of disjoint $A, B \subseteq V(G)$ such that $e(A,B,\{v\})\leq k$ for every $v\in V(G)$,
    each event $U\in \U_p$ changes $|E(G_X[A, B])|$ by at most $2k$.
    Therefore, by Azuma's inequality (Lemma~\ref{lem_azuma}), the probability that $|E(G_X[A, B])|$ deviates from its expectation by more than $\varepsilon p|A||B|$ is at most $2 \exp(-\varepsilon^2 p^2 |A|^2 |B|^2 / 4n k^2)$.

    The number of pairs $(A,B)$ such that $n^{2/5} \leq |A|, |B| \leq n^{1/2}$ is at most 
    \begin{align*}
     \left(n^{1/2}\binom{n} {n^{1/2}}\right)^2 \leq n \left(\frac{en}{n^{1/2}}\right)^{2n^{1/2}}\leq (en^{1/2})^{2n^{1/2}+2}. 
    \end{align*}
    Among these choices, we call a pair $(A,B)$ \emph{bad} if $\big||E(G_X[A, B])|-p|A||B|\big| \geq \varepsilon p|A||B|$.
     Then the probability that there exists a bad pair $(A, B)$ is at most
    \begin{align*} 
     2 \exp(-\varepsilon^2 p^2 |A|^2 |B|^2 / 4nk^2)(en^{1/2})^{2n^{1/2}+2} & \leq 2\exp(-\varepsilon^2 p^2 |A|^2 |B|^2 / 4 n k^2 + 10n^{1/2}\log n)\\
     & \leq 2\exp(-\varepsilon^2 p^2n^{3/5}/4 k^2 + 10n^{1/2}\log n) \\
     &  = o\Big(\frac{1}{n}\Big),
    \end{align*}
    provided $k=(\log n)^{10}$ and $\varepsilon,p\geq (\log n)^{-100}$.
    If $X, A, B$ are pairwise disjoint, then $e(X, A, B)=|E(G_X[A, B])|$, which completes the proof.
\end{proof}

By using the above lemma, we prove a concentration result for larger sets $A$ and $B$ of size $o(n)$.

\begin{lemma}\label{lem_edges_largeset}
 Let $X\sim \mathcal{U}_p$ for $p \geq 1/(\log n)^{100}$ and let $\eta \geq 1/(\log n)^{100}$. Then with probability $1-o(1/n)$, every pair of disjoint subsets $A, B \subseteq V(G) \setminus X$ of size at least $\eta n$ satisfies
 $e(A, B, X) = (1 \pm \eta)p|A||B|$.
\end{lemma}
\begin{proof}
    Let us fix disjoint subsets $A, B \subseteq V(G) \setminus X$ of size at least $\eta n$. We will partition $A$ and $B$ into ``manageable'' subsets to which \Cref{lem_edges_smallset} can apply.
    To that end, let $s:=\lceil\frac{|A|}{3n^{2/5}} \rceil$ and $t:=\lceil\frac{|B|}{3n^{2/5}} \rceil$.
    Partition $A$ into $s$ subsets $A_1, A_2, \ldots, A_s$ by choosing $i_a \in [s]$ uniformly and independently at random for each $a \in A$ and putting $a$ in $A_{i_a}$. We also partition $B$ in the same way but using $t$ subsets $B_1, B_2,\ldots, B_t$.
    For each $i\in [s]$ and $j\in [t]$, the expected sizes of each $A_i$ and $B_j$ are $|A|/s$ and $|B|/t$, respectively, each of which are of size $(3+o(1))n^{2/5}$.
    The Chernoff bound (\Cref{lem_chernoff}) then implies that all $A_i$ and $B_j$ are of size between $n^{2/5}$ and $n^{1/2}$ with probability at least
     $1-2(s+t)\exp(-n^{2/5}/12) = 1-o(1/n)$.

    For fixed $v \in V(G)$, $i \in [s]$, and $j \in [t]$, let $\mathcal{E}_e$ be the event that an edge $e \in E(G)$ containing~$v$ is in $G[\{v\}, A_i, B_j]$. Then
    $\mathcal{E}_e$ occurs with probability at most $1/(st) = o(1/n)$ for any~$e \in E(G)$ containing $v$.
    Furthermore, linearity of $G$ ensures that the events $\{\mathcal{E}_e\}_{e\in E(G), e \ni v}$ are mutually independent.
    Thus, by the one-sided Chernoff bound (Lemma~\ref{lem_oneside_chernoff}), $e(A_i, B_j, \{v\}) \leq (\log n)^{10}$ with probability at least $1-\exp(-(\log n)^8)$. Together with the union bound, this is enough to conclude that $e(A_i, B_j, \{v\}) \leq (\log n)^{10} $ holds for all vertices $v \in V(G)$ and indices $i \in [s]$, $j \in [t]$ with positive probability.
    %at least $1-\exp(-(\log n)^7)$.
    Therefore, there exist partitions $A_1, A_2, \ldots, A_s$ of $A$ and $B_1, B_2, \ldots, B_t$ of $B$ such that $A_i$ and $B_j$ satisfy the conditions~\eqref{6.1.1} and~\eqref{6.1.2} of  \Cref{lem_edges_smallset} for every $i$ and $j$.
    
    Thus, Lemma~\ref{lem_edges_smallset} implies that, with probability at least $1-o(1/n)$,
    we have $e(A_i, B_j, X) =(1 \pm \eta) p |A_i||B_j|$ for every $i \in [s], j \in [t]$ (for all choices of $A, B\subseteq V(G) \setminus X$).
    Now, taking the sum over all $i \in [s]$ and $j \in [t]$ completes the proof of the lemma.
\end{proof}

Finally, using Lemma~\ref{lem_edges_largeset}, we find a matching $M$ that covers almost all the vertices of any prescribed set $A$ which is small enough and disjoint from the random set $X$.

\begin{lemma}\label{lem_embedding_matching}
 Let $\varepsilon>0$ be a constant and suppose $0< 1/n \ll \varepsilon$.
 Let $X\sim \mathcal{U}_p$ for $p \geq {1}/{(\log n)^{20}}$.
 Then the following holds with probability $1-o(1/n)$.
 For all vertex sets $A \subseteq V(G) \setminus X$ of size at most $pn/2$, there exists a matching $M$ of size at least $|A|-\varepsilon p n$ such that  each edge $e \in M$ satisfies $|e \cap X|=2$ and $|e \cap A|=1$.
\end{lemma}
\begin{proof}
Choose $\tau>0$ such that $0<1/n \ll\tau\ll\varepsilon$ and let $\eta = \tau p^2$. In particular, we choose $\tau$ so that $\eta =\tau p^2 \geq 1/(\log n)^{100}$.

Let $\{X_1, X_2\}$ be a random partition of $X$ obtained by assigning each $U_i \in \U_p$ to exactly one of $X_1, X_2$ independently at random.
 Then each~$X_i$ is a (possibly dependent) copy of $\U_{p/2}$ i.e., $X_i\sim \mathcal{U}_{p/2}$.

 By Lemma~\ref{lem_edges_largeset}, with probability $1-o(1/n)$, every pair of disjoint subsets $A', B' \subseteq V(G) \setminus X_2$ of size at least $\eta n$ satisfies $e(A', B', X_2) = (1\pm \eta) p|A'||B'|/2$.
 Similarly, with probability $1-o(1/n)$, every pair of disjoint subsets $A', B' \subseteq V(G) \setminus X_1$ of size at least $\eta n$ satisfies $e(A', X_1, B') = (1\pm \eta) p|A'||B'|/2$.
 Conditioning on these two events, say $\mathcal E_1$ and $\mathcal E_2$, the following holds for any $A^* \subseteq V(G) \setminus X$ with $|A^*| = pn/2$:
 \begin{enumerate}[(i)]
     \item For every $A' \subseteq A^*$ and $X' \subseteq X_1$ of size at least $\eta n$, we have $e(A', X', X_2)=(1 \pm \eta) p|A'||X'|/2$;
     \item For every $A' \subseteq A^*$ and $X' \subseteq X_2$ of size at least $\eta n$, we have $e(A', X_1, X')=(1 \pm \eta) p|A'||X'|/2$.
 \end{enumerate}
 This implies that $e(A^*, \{v\}, X_2) =  (1 \pm \eta) p|A^*|/2$ for all but at most $2\eta n$ vertices $v$ in $X_1$. 
 Indeed, if not, we can collect $\eta n$ vertices to obtain a subset $X' \subseteq X_1$ with $|X'| \geq \eta n$ and either $e(A^*, X', X_2) > (1 +\eta) p|A^*||X'|/2$ or $e(A^*, X', X_2) < (1 -\eta) p|A^*||X'|/2$, which contradicts the conditioned events $\mathcal E_1$ and $\mathcal E_2$.
 Moreover, by swapping the roles of $X_1$ and $X_2$, one can prove that, 
 for all but at most $2\eta n$ vertices $v\in X_2$, $e(A^*, X_1, \{v\}) = (1 \pm \eta) p|A^*|/2$ holds.
  Analogously, for all but at most $2\eta n$  vertices $v\in A^*$,
 $e(\{v\}, X_1, X_2) = (1 \pm \eta) p|X_1|/2$ holds.

One can also control the sizes of $X
_1$ and $X_2$. Namely, since $X_1$ and $X_2$ are 2-Lipschitz, Azuma's inequality (Lemma~\ref{lem_azuma}) gives that $|X_1|, |X_2|=(1\pm \eta) p n/2$ with probability at least $$1-2\exp\Big(-\frac{\eta^2 p^2 n^2 }{16n}\Big) \geq  1-2\exp\Big(-\frac{\tau^2 p^6 n}{16}\Big) = 1-o\Big(\frac{1}{n}\Big).$$
In particular, for each $i \in [2]$, we have $(1 \pm \eta) p|X_i|/2 = (1 \pm 3\eta)p^2 n/4$.

Assuming all these four high probability events occur, namely $\mathcal E_1, \mathcal E_2$, and that $|X_1|, |X_2|=(1\pm \eta) p n/2$, let $A \subseteq V(G) \setminus X$ be a vertex set of size at most $pn/2$. By adding arbitrary vertices to $A$, we obtain a set $A^*$ with $|A^*|=pn/2$.
 Let $H$ be the $3$-graph obtained by removing all the exceptional vertices from $G[A^*, X_1, X_2]$, i.e., those vertices whose degrees in $G[A^*, X_1, X_2]$ are not in the range $(1 \pm 3 \eta)p^2 n/4$.
 By the discussion above concerning the number of such exceptional vertices, we remove at most $2 \eta n$ vertices from each part. Thus, as $G$ is linear, each remaining vertex has degree $$\frac{1}{4}(1 \pm 3\eta)p^2 n \pm 4 \eta n = \frac{1}{4}(1 \pm (3\eta +16\eta/p^2)) p^2 n.$$
 As $H$ is linear, each pair of vertices has codegree at most $1$.
 Therefore, by Pippenger's theorem (\Cref{lem_Pippenger}), as $(3\eta +16\eta/p^2) \ll \varepsilon/2$, $H$ contains a matching $M^*$ of size at least \begin{align*}
     (1-\varepsilon/2)(|A^*|+|X_1|+|X_2|-6\eta n)/3 \geq (1-\varepsilon/2)(|A^*|-3 \eta n) \geq |A^*|-\varepsilon p n. 
 \end{align*}
By removing the edges containing a vertex of $A^*\setminus A$  from $M^*$, we obtain the desired matching $M$ of size at least $|A|-\varepsilon pn.$
\end{proof}

Lemma~\ref{lem_embedding_matching} allows us to extend $T_{i-1}$ to $T_i$ whenever $T_i$ is obtained by adding a matching, each of whose edges contains exactly one vertex in $T_{i-1}$.
We now proceed to the other case where we add $o(n)$ bare paths of length 3.
A pair of $u$--$v$ paths $P_1$ and $P_2$ are said to be \emph{internally vertex-disjoint} if their vertex sets are disjoint except the two vertices $u$ and $v$ i.e., $(V(P_1) \setminus \{u, v\}) \cap (V(P_2) \setminus \{u, v\}) = \emptyset$.
The next lemma proves that there are ``many" internally vertex-disjoint paths of length $3$ between pairs of vertices $u$ and $v$.

\begin{lemma}\label{lem_embedding_bare_path}
 Let $0<1/n \ll \mu \ll p$ and let $X\sim \mathcal{U}_p$. Then with probability $1-o(1/n)$, the following holds:
 For every pair of distinct vertices $u, v \in V(G)$, there exist at least $\mu n$ internally vertex-disjoint $u$--$v$ paths of length $3$ such that all of their internal vertices are contained in $X$.
\end{lemma}

To prove this, we need the following consequence of concentration inequalities.

\begin{lemma}\label{lem_rand_deg}
Let $0<1/n \ll p$ and let $X\sim \mathcal{U}_p$. Then with probability $1-o(1/n)$, every vertex $u \in V(G)$ satisfies $e(\{u\},X,X)\geq p^2n/3$.
\end{lemma}
\begin{proof}
 For a fixed vertex $u \in V(G)$, let $Y_u := e(\{u\}, X, X)$.
 An edge $e$ containing $u$ contributes to $Y_u$ if and only if $e\setminus \{u\} \in \U_p$ or the two vertices in $e\setminus \{u\}$ are in two disjoint sets in $\U_p$.
 Thus, $e$ is an edge in $G[\{u\},X,X]$ with probability either $p$ or $p^2$, which implies that $p^2(n-1)/2 \leq \mathbb{E}[Y_u] \leq p(n-1)/2$.
 As each event $U_i \in \U_p$ affects at most two edges containing $u$, $Y_u$ is $2$-Lipshitz. This enables us to apply Azuma's inequality (Lemma~\ref{lem_azuma}) to obtain
 \begin{align*}
     \mathbb{P}\left(Y_u < \frac{p^2n}{3}\right) \leq 2\exp \Big(-\frac{p^4 n^2}{10^3 n}\Big) \leq \frac{1}{n^3}.
 \end{align*}
 Thus, the probability that $Y_u\geq p^2n/3$ holds for every $u\in V(G)$ is at least $1-1/n^2=1-o(1/n)$.
\end{proof}

\begin{proof}[Proof of Lemma~\ref{lem_embedding_bare_path}]
 Let $\{X_1, X_2, X_3\}$ be a random partition of $X$ obtained by assigning each $U_i \in \U_p$ to exactly one of $X_i$ independently at random.
 Then each~$X_i$ is a (possibly dependent) copy of $\U_{p/3}$.
 
 By Lemma~\ref{lem_rand_deg}, with probability $1-o(1/n)$, all vertices $u, v \in V(G)$ satisfy $e(\{u\}, X_1, X_1) \geq p^2n/27$ and $e(\{v\}, X_3, X_3) \geq p^2 n/27$. 
 By Lemma~\ref{lem_edges_largeset}, with probability $1-o(1/n)$, for every pair of disjoint subsets $A,B\subseteq V(G) \setminus X_2$ with $|A|,|B|\geq p^2n/100$, we have $e(A,X_2,B) \ge p|A||B|/6$.
  As both of these two events hold simultaneously with probability $1-o(1/n)$, it suffices to show that these two events imply the existence of the desired paths for all pairs of distinct vertices $u,v \in V(G)$.
 
Suppose for a contradiction that there exists a pair $u,v$ of distinct vertices such that there are less than $\mu n$ internally vertex-disjoint $u$--$v$ paths of length three (with their internal vertices contained in $X$). Choose a maximal collection $P_1, P_2, \cdots, P_k$ of internally vertex-disjoint $u$--$v$ paths of length three such that their internal vertices are contained in $X$; then $k< \mu n$.
 Let $Y_i$ be the set of the internal vertices of $P_i$ for $1 \le i \le k$ (so the sets $Y_i$ are pairwise disjoint).
 Let $Y:=\bigcup_{i=1}^k Y_i$, and let $E_1 := E(G[\{u\}, X_1, X_1]) - E(G[\{u\}, Y, X_1])$ and let $E_3 := E(G[\{v\}, X_3, X_3]) -E(G[\{v\}, Y, X_3])$.  As $e(\{u\}, Y, X_1) \le |Y|$ and $e(\{v\}, Y, X_3) \le |Y|$  (since $G$ is linear) and $|Y|\leq 5k \le 5\mu n$, by the conditioned events, we have $|E_1|, |E_3| \ge p^2n/27 - 5\mu n \geq p^2n/30$.

 Let $Z_1$ and $Z_3$ be the set of vertices in the edges of $E_1$ and $E_3$ except $u$ and $v$, respectively.
 Then $|Z_1|, |Z_3| \geq p^2 n/15$.
 By the conditioned events, we have $e(Z_1, X_2, Z_3) \ge p|Z_1||Z_3|/6 \geq p^5 n^2 /2000$.
 As there are at most $5\mu n^2$ edges incident to a vertex in $Y$, $e(Z_1, X_2\setminus Y, Z_3) \geq p^5 n^2 /2000 -  5 \mu n^2$.
 In particular, there is an edge $e$ in $G[Z_1, X_2\setminus Y, Z_3]$.
 Hence, there exists a $u$--$v$ path $P'$ of length three containing $e$ such that $V(P')$ is disjoint from $Y$, i.e., $P'$ is internally vertex-disjoint from the collection $P_1, \cdots, P_k$. Moreover, the internal vertices of $P'$ are contained in $X$. This contradicts the maximality of the collection $P_1, \cdots, P_k$. Hence, the conditioned events, which hold with probability $1-o(1/n)$, imply the existence of the desired collection of paths for all pairs $u,v \in V(G)$ of distinct vertices.
\end{proof}

\section{Stars}\label{sec_stars}
In this section, we will prove that if the minimum degree of a linear hypergraph $G$ is large enough, then we can find vertex-disjoint stars of desired size.

\begin{lemma}\label{lem_embedding_stars}
Let  $0<1/n\ll \varepsilon< \frac{1}{100}$.
 Let $G$ be an $n$-vertex linear $3$-uniform hypergraph with $\delta(G) \geq (1-\varepsilon)n/2$, and let $X=\{v_1, v_2, \ldots, v_\ell\}$ be an independent set in $G$ of size $\ell \leq \varepsilon^2 n/10$. 
 Then for any given positive integers $n_1, n_2, \ldots, n_{\ell}$ such that $\sum_{i=1}^\ell n_i \leq (1-5 \varepsilon)n/2$, $G$ contains 
 vertex-disjoint stars $S_1, S_2, \ldots, S_\ell$ such that each $S_i$ has size $n_i$ and is centered at $v_i$.
\end{lemma}
\begin{proof}
 Let $m=\lceil  \varepsilon^{-1} \rceil+1$. 
 As $G$ is a linear $3$-graph with $\delta(G)\geq (1-\varepsilon)n/2$, there exists a vertex set $W$ of size at most $\varepsilon n$ such that for all vertices $w\in V(G)\setminus (W\cup \{v_1\})$, there exists an edge $v_1ww'$ in $G$ for some $w'\in V(G)\setminus (W\cup \{v_1\})$.
 
 We then choose  vertex-disjoint stars $S_1, \cdots, S_\ell$ in $G$ which satisfy:
 \begin{enumerate}[(i)]
  \item each $S_i$ is centered at $v_i$ and has size $t_i \leq n_i + m$. \label{condition_1} and
  \item among all $S_1,S_2,\cdots,S_\ell$ that satisfy \eqref{condition_1}, choose one that maximizes $\sum_{i=1}^\ell t_i$.\label{max_t}
 \end{enumerate}
 By allowing empty $S_i$'s, such a choice is always possible.
 We claim that $t_i \geq n_i$ for all $i \in [\ell]$, which concludes the proof.
 Suppose to the contrary that $t_1<n_1$, by reindexing if necessary.
 
 For two vertex sets $U,U'$, denote by $N(U,U')$ the set of vertices $v$ such that there is an edge $uu'v$ of $G$ with $u\in U$ and $u'\in U'$.
 Let $M$ be the graph matching on $V(G)$ where $E(M) = \bigcup_{i\in [\ell]}\{xy : xyv_i \in E(S_i)\}.$ Let $N_M(U)= \bigcup_{u\in U}\{x: xu\in M\}$.
 These definitions are convenient for us to describe how one can replace the stars $S_1,\dots, S_\ell$ with other stars contradicting the maximality \eqref{max_t}.
 Let $A_1:=V(G)\setminus\bigcup_{i=1}^\ell V(S_i)$.
 We then define $A_2, \cdots, A_m, B_1, \cdots, B_m \subseteq V(G)$ recursively as
 \begin{align*}
     B_i := N(A_i,\{v_1\}) ~\text{ and }~ A_{i+1}:=N_M(B_i).
 \end{align*}
   As $|A_1| \geq n-2\sum_{i=1}^\ell (n_i + m) \geq 2\varepsilon n$ and $v_1$ has at least $(1-\varepsilon)n$ neighbors, $v_1$ has at least $\varepsilon n$ neighbors in $A_1$. Thus, $|B_1| \geq \varepsilon n$.
 
 From the definition of $B_i$ and the fact that $G$ is linear, for each $y\in B_i$, we have unique $x\in A_i$ with $xyv_1\in E(G)$. Again, for each $x \in A_i$ with $i>1$, there exists unique $y'\in B_{i-1}$ such that $y'x \in M$.
  Hence, for each vertex $v\in \bigcup_{i\in [\ell]}B_i$, there exists a unique sequence $v=y_k,x_k,y_{k-1},x_{k-1},\dots, y_1,x_1$ such that $\{x_i\}= N(\{y_i\},\{v_1\})$ and $\{y_{i-1}\}= N_M(\{x_i\})$ for all $i\in [\ell]$. As this sequence is unique, this determines the unique $k\in [\ell]$ that $v\in B_k$. From this we conclude that 
  \begin{align*}
     A_i\cap A_j=\varnothing \enspace \text{and} \enspace B_i\cap B_j = \varnothing \text{ for } i\neq j \in [m].
  \end{align*}
 Furthermore, for such a sequence for every vertex $v\in \bigcup_{i\in [\ell]}B_i$, the vertices $y_k,x_k,\dots, y_1,x_1$ are all distinct. As all $A_i$s are pairwise disjoint and all $B_i$s are pairwise disjoint, we only have to show that $x_i\neq y_j$ for all $i\neq j\in [k]$. If not, choose $i$ and $j$ with the minimum $|i-j|>0$ such that $x_i=y_j$.
 Then the definition ensures that $\{x_j\}=\{y_i\} = N(\{x_i\},\{v_1\})$, so
 $x_i=y_j$ and $x_j=y_i$. By symmetry, assuming $i>j$, we know that $i>j+1$.
 If not, then $i=j+1$, but $y_j= x_i=x_{j+1}=N_M(y_j)$, a contradiction that $M$ forms a matching. Moreover, as we have $y_{i-1}=x_{j+1}= N_M(x_i)$ and $i>j+1$, a contradiction to the the minimality of $|i-j|$.

For each $i\in [\ell]$, let $L_i$ be the set of all leaf vertices of $S_i$ and let $L = \bigcup_{i\in [\ell]} L_i$.
We claim that, for  $k\in [m]$, the vertex set $B_k$ is contained in $L$. 
 
 Suppose that $B_k \nsubseteq L$ while $B_{j}\subseteq L$ for all $j<k$.
Choose $y_k\in B_k \setminus L$, and consider the unique sequence $y_k,x_k,
\dots, y_1,x_1$ as above. As all vertices in the sequence are distinct, we use this sequence to contradict the maximality assumption of $\sum_{j=1}^{\ell} t_j$ in~\eqref{max_t}.
 Delete all the edges of the form of $y_ix_{i+1}v_{i'}$ from $S_{i'}$ for $1 \leq i \leq k-1$ and add all edges of the form of $x_iy_iv_1$ for $1 \leq i \leq k$ to $S_1$. Note that the choice of $i'$ is uniquely determined for each $i$.
 As $x_1, y_k$ are not contained in $\bigcup_{j=1}^\ell V(S_j)$ and all $x_i, y_i$ are distinct, $S_1, \cdots, S_\ell$ are still vertex disjoint stars. Moreover, $S_1$ has at most $n_1 + k \leq n_1 + m$ edges.
 Thus, our new $S_1, \cdots, S_\ell$ satisfies the condition \eqref{condition_1}
 and increases $\sum_{j=1}^\ell t_j$  by $1$.
 %Then it is a contradiction as $\sum_{j=1}^\ell t_j$ increased by $1$.
 Therefore, $B_k \subseteq \bigcup_{j=1}^\ell V(S_j)$.

 \begin{figure}
     \centering
     \includegraphics[width=\textwidth]{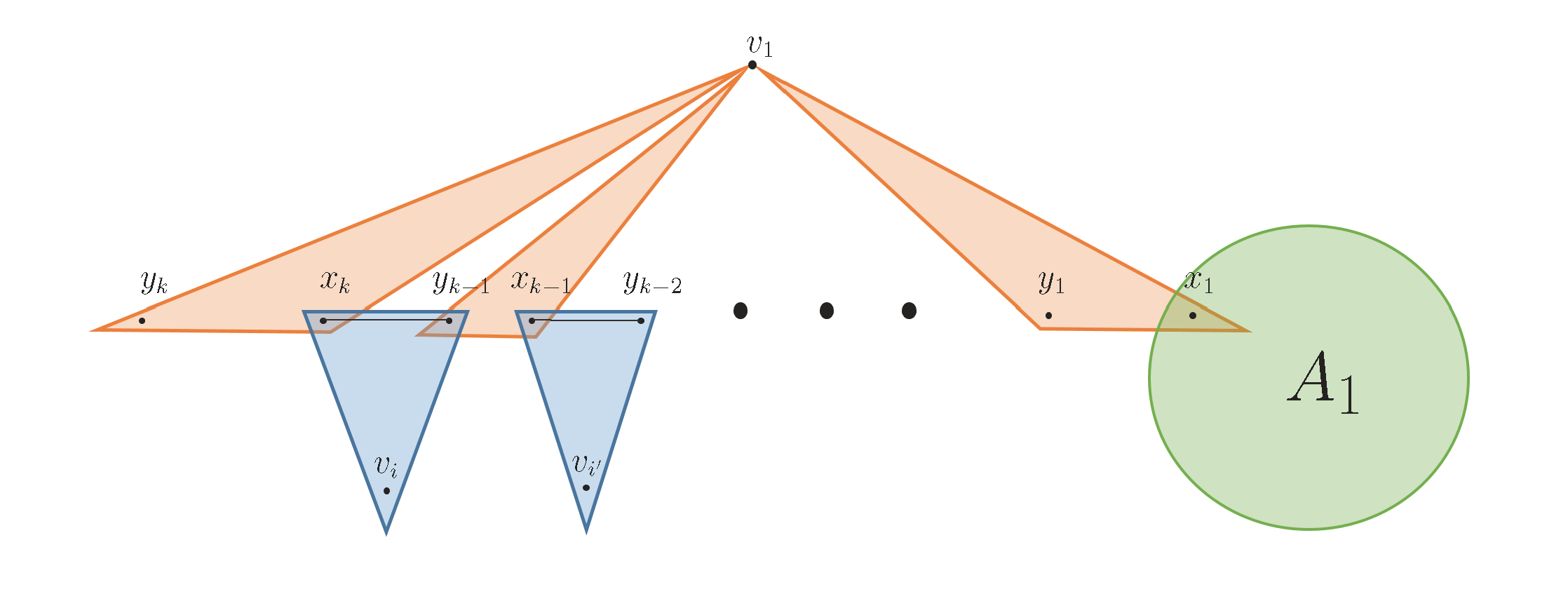}
     \caption{Choice of $x_1, y_1, \ldots, y_k=v$}
     \label{fig:alternating_seq_3}
 \end{figure}
 
As $M$ is a perfect matching on $L$, we have $|A_{i+1}|=|B_i|$ for all $i\in [m]$.
 On the other hand, by the definition of $B_i$, we have $|B_i| = |A_i \setminus W|$ for all $i\in [m]$. 
As we have $|A_1| \geq n-2\sum_{i=1}^\ell (n_i + m) \geq 2\varepsilon n$ and all $A_i$ are pairwise disjoint, we conclude that for each $k$
 \begin{align*}
     \sum_{i\in [k]} |B_i| = \sum_{i\in [k]} |A_{i}\setminus W|
     &\geq |A_1| + \sum_{2\leq i\leq k} |A_i|  - |W|
     \geq 2\varepsilon n + \sum_{1\leq i\leq k-1} |B_i| - \varepsilon n \geq \varepsilon n + \sum_{1\leq i\leq k-1} |B_i|.
 \end{align*}
This yields that 
$$\Big\vert\bigcup_{i\in [m]} B_i \Big\vert= \sum_{i\in [k]} |B_i|
 \geq m \varepsilon n  > n,$$ %a contradiction 
 which concludes the proof by contradiction.
%
%
%%%%%%%%%%%%%%%
%
%
\end{proof}

\section{Proof of \Cref{main_thm}}\label{sec_combine}
We begin by stating a variant of the well-known fact which easily follows from a simple greedy algorithm. We omit the proof.
\begin{lemma}\label{lem_small_forest}
    Let $G$ be a linear hypergraph with minimum degree $\delta$.
    Then $G$ contains every tree $T$ with less than $\delta/2$ vertices.
\end{lemma}

We may assume that $\varepsilon$ is small enough. Choose $\mu$ and $n$ so that we have $0<1/n \ll \mu \ll \varepsilon \ll 1$. 
Let $T$ be a hypertree with at most $(1-\varepsilon)n$ vertices.
Applying Lemma~\ref{lem_tree_split} with $D= \lceil \log^{10} n \rceil$ and the chosen $\mu$, we obtain a chain of sub-hypergraphs $T_0 \subseteq T_1 \subseteq \cdots \subseteq T_\ell = T$ with $\ell \leq 10^5 D \mu^{-2}$, $s \in [\ell]$ satisfying the assertions \eqref{it:T0}--\eqref{it:path} of Lemma~\ref{lem_tree_split}. 

By~\eqref{it:T0}, $|V(T_0)| \leq 3\mu n$, 
so there exists an injective homomorphism $\varphi:V(T_0) \rightarrow V(G)$ that embeds~$T_0$ into $G$ by Lemma~\ref{lem_small_forest}. 
In what follows, we shall identify $T_0$ as its image under $\varphi$, i.e., we assume that $T_0$ is embedded as a subgraph of $G$.
Let $\{v_1, \cdots, v_k\} \subseteq V(G)$ be the set of vertices of $T_0$ with $e_{T_1}(v_i, V(T_1 \setminus T_0), V(T_1 \setminus T_0)) \geq D$ and let $d_i:=e_T(v_i, V(T_1 \setminus T_0), V(T_1 \setminus T_0))$ for each $i \in [k]$. Note that~$k$ is at most $ O(n (\log n)^{-10}) \leq \mu n \leq \varepsilon^2 n/10$.

First we remove all vertices in $V(T_0) \setminus \{v_1, \cdots, v_k\}$ from $G$.
This removes at most $3\mu n$ edges incident to a fixed remaining vertex $x\in (V(G) \setminus V(T_0)) \cup \{v_1, \cdots, v_k\}$. 
Furthermore, we remove all the edges containing at least two of $v_1, \cdots, v_k$ from $G$. This removes at most $k/2$ edges incident to any vertex $x \in V(G)$.
Thus, after removing all these edges from $G$, the minimum degree of the resulting hypergraph $H$ is still at least $(1-10 \mu)n/2 \geq (1-10 \mu)|V(H)|/2$.
Let $d:=\sum_{i=1}^k d_i$ and let $n_i := \lceil\frac{d_i}{2d} n (1-\varepsilon/8)\rceil$.
Then $\sum_{i=1}^{k} n_i \leq k+\sum_{i=1}^{k}\frac{d_i}{2d} n (1-\varepsilon/8) \leq  (1-\frac{\varepsilon}{10}) \frac{|V(H)|}{2}$ as $|V(H)| \geq n-3\mu n$.
By applying Lemma~\ref{lem_embedding_stars} to $H$, we get vertex-disjoint stars $S_1, \cdots, S_k$ in $G$ where each $S_i$ is centered at $v_i$ of size $d_i$ and does not contain any vertices of $T_0$ other than $v_i$.

We then define a singleton-pair partition $\U$ of $V(G)$ as follows. For every $1 \le i \le k$ and for every edge $xyv_i \in E(S_i)$, let $\{x, y\}$ be a part of $\U$. Moreover, let each of the remaining vertices of $V(G)$ be a part of $\U$ of size one.

Let $p_{0} :=\varepsilon/200$. Let $m_i := |V(T_i) \setminus V(T_{i-1})|$ and let $p_i := (1+\frac{\varepsilon}{4})\frac{m_i}{n} + \frac{\varepsilon}{4\ell}$ for $i \in [\ell]$.
Recall that $\ell \leq 10^5 D\mu^{-2}= O((\log n)^{10}) $.
As $\varepsilon>0$ is a constant, $p_i \geq \frac{1}{(\log n)^{20}}$ for all $i \in [\ell]$.

Now sample $X_1 \subseteq V(G)$ by choosing each $U\in\U$ with probability $p_1$ independently at random and taking their union, i.e., $X_1\sim \U_{p_1}$. 

Provided that $X_1, X_2, \cdots, X_{i-1}$ are chosen,
select each $U\in\U$ that is not included in $X_1\cup\cdots\cup X_{i-1}$ with probability $p_i/(1-p_1-p_2-\cdots - p_{i-1})$ independently at random and add it to $X_i$. 
After choosing $X_\ell$, we choose the ``reservoir" $R$ by selecting each $U\in\U$ that is not in $X_1 \cup \cdots \cup X_{\ell}$ with probability $p_0/(1-p_1-p_2-\cdots - p_{\ell})$ independently at random and adding it to $R$. 
Note that $\sum_{i=0}^\ell p_i \leq (1+\varepsilon/4)|T|/n + \varepsilon/4 + \varepsilon/200 \leq 1$.
Then $X_i \sim \U_{p_i}$ for each $i \in [\ell]$ and $R \sim \U_{p_0}$.

Let $\mathcal{F}, \mathcal{E}_0, \mathcal{E}_1, \mathcal{E}_M$ be the following events.
\begin{enumerate}
    \item[$\mathcal{F}$:] $|V(T_0) \cap X_i| \leq 10p_i \mu n$ for each $i \in [\ell]$ and $|V(T_0) \cap R| \leq 10p_0 \mu n$.
    \item[$\mathcal{E}_0$:] Every vertex $v$ is incident to at least $10^4 \mu n$ edges $e$ such that $e \setminus \{v\} \subseteq R$ and between every pair $u, v \in V(G)$ of distinct vertices, there are at least $100 \mu n$ internally vertex-disjoint $u$--$v$ paths of length 3 and all of their internal vertices are in $R$.
    \item[$\mathcal{E}_1$:] For each $i \in [k]$, $e(\{v_i\}, X_1, X_1) \geq d_i$.
    \item[$\mathcal{E}_M$:] For each $2\leq i\leq \ell$ and for all sets $A \subseteq V(G) \setminus X_i$ of size $|A| \leq p_i n/2$, there exists a matching $M$ of $G$ of size at least $|A|-\mu p_i n$ such that each edge $e \in M$ satisfies $|e \cap X_i|=2$ and $|e \cap A|=1$.
\end{enumerate}

We first show that all of the events $\mathcal{F}, \mathcal{E}_0, \mathcal{E}_1, \mathcal{E}_M$  hold with high probability and then conditioning on these events, we show that one can embed $T$ into $G$. 
\begin{claim} 
\label{highprobevents}
    $\mathbb{P}(\mathcal{E}_0), \mathbb{P}(\mathcal{E}_1), \mathbb{P}(\mathcal{F}) = 1-o(1)$, and $\mathbb{P}(\mathcal{E}_M) = 1-o(1)$.
\end{claim}
\begin{proof}
Firstly, \Cref{lem_embedding_bare_path,lem_rand_deg} imply that $\mathbb{P}(\mathcal{E}_0) = 1-o(1/n)$.
Secondly, by Azuma's inequality (Lemma~\ref{lem_azuma}) we have
$$\mathbb{P}(\mathcal{F}^c) \leq \sum_{i=0}^\ell 2\exp\left(-\frac{p_i^2 \mu^2 n^2}{10n}\right) \leq o\Big(\frac{1}{n}\Big),$$
where the last inequality follows from the fact that $p_i\geq \frac{\varepsilon}{4\ell} \geq \frac{1}{(\log n)^{20}}$ for each $i$.
Now we show $\mathbb{P}(\mathcal{E}_1) = 1-o(1)$.
As $T_1\setminus T_0$ consists of the stars $S_1,\cdots,S_k$ that have $2d$ leaves in total, $m_1=|V(T_1)\setminus V(T_0)|=2d$.
Hence, the expected value of $e(\{v_i\}, X_1, X_1)$ is
\begin{align*}
    p_1 n_i \geq \left(1+\frac{\varepsilon}{4}\right)\frac{m_1}{n}
    \cdot \left(1-\frac{\varepsilon}{8}\right) \frac{d_in}{2d}
    \geq \left(1+\frac{\varepsilon}{4}\right)\frac{2d}{n}
     \cdot \left(1-\frac{\varepsilon}{8}\right) \frac{d_in}{2d}  \geq \left(1+\frac{\varepsilon}{10}\right) \cdot d_i.
\end{align*}
Indeed, for any edge in $S_i$, the probability that its two leaves are included in $X_1 \sim \mathcal U_{p_1}$ is equal to $p_1$, as they form a part (of size two) in the singleton-pair partition $\mathcal U$.
Moreover, these events for all edges in $S_i$ are mutually independent. This enables us to use the Chernoff bound (\Cref{lem_chernoff}) to prove that $e(\{v_i\}, X_1, X_1)$ is less than $2d_i$ with probability at most 
$$2\exp\Big(-\frac{\varepsilon^2 d_i}{10^3}\Big) \leq \exp\Big(-\frac{\varepsilon^2 \log^{10} n}{10^3}\Big) = o\Big(\frac{1}{n}\Big),$$
where the inequality follows from $d_i\geq D\geq \log^{10}n$ which is guaranteed by~\eqref{it:T1} of Lemma~\ref{lem_tree_split}.
Therefore, $\mathbb{P}(\mathcal{E}_1) = 1-o(1)$. Finally, \Cref{lem_embedding_matching} together with a union bound ensures that $\mathbb{P}(\mathcal{E}_M) = 1-o(\frac{\ell}{n}) = 1-o(1)$.
\end{proof}
By Claim~\ref{highprobevents}, the events $\mathcal{F}, \mathcal{E}_0, \mathcal{E}_1$, and $\mathcal{E}_M$ occur with probability $1-o(1)$.
Now we show that there exists a sequence $T_0\subseteq T_1\subseteq \dots \subseteq  T_\ell = T$ of subgraphs of $G$
satisfying the following properties for all $i \in [\ell]$ (conditioning on the events $\mathcal{F}, \mathcal{E}_0, \mathcal{E}_1$, and $\mathcal{E}_M$).
\begin{enumerate}[(a)]
    \item \label{algoa} If $i\neq s+1$, then $T_i$ extends $T_{i-1}$ and it satisfies $V(T_i)\setminus V(T_{i-1}) \subseteq X_i \cup R$ and $|(V(T_i)\setminus V(T_{i-1})) \cap R| \leq 30 p_i \mu n$. 
    %Choose one such choice of $T_i$.
    \item \label{algob} If $i=s+1$, then $T_{s+1}$ extends $T_{s}$ and it satisfies $V(T_{s+1})\setminus V(T_{s}) \subseteq R$ and 
    $|V(T_{s+1})\setminus V(T_{s})|\leq 7\mu n$.
\end{enumerate}

As $T_0$ vacuously satisfies the above two properties, assume that we have $T_0\subseteq \dots \subseteq T_{i-1}$ satisfying the above properties with the maximum $i$, and assume $i\leq \ell$.
Since $T_0\subseteq \dots \subseteq T_{i-1}$ satisfy \eqref{algoa} and \eqref{algob}, we have 
\begin{align}\label{eq: TR size}
    |V(T_{i-1})\cap R| \leq |V(T_0)| + \sum_{j<i,j\neq s} 30 p_j \mu n + 7\mu n \leq 40 \mu n.
\end{align}

If $i=1$, by using vertices in $X_1 \cap S_j$ for each $j \in [k]$, one can extend $T_0$ to $T_1$ by $\mathcal{E}_1$. It is straightforward to check $V(T_1)\setminus V(T_{0}) \subseteq X_1 \cup R$ and $|(V(T_1)\setminus V(T_{0})) \cap R|=0 
$, as we only used vertices in $X_1$ to extend $T_0$ to $T_1$.
Hence, \eqref{algoa} is satisfied, a contradiction to the maximality of $i$. Hence we have $i>1$.

Suppose $2 \leq i  \le \ell$ and $i\neq s+1$.
By~\eqref{it:matching} of Lemma~\ref{lem_tree_split}, in this case, $T_i$ can be obtained by adding a matching to $T_{i-1}$ such that $V(T_i) \setminus V(T_{i-1})$ is a matching leaf set of $T_i$. Let $A_i$ be the set of vertices of $T_{i-1}$ 
that are contained in the edges of the matching $E(T_i)\setminus E(T_{i-1})$.
Then $|A_i| = m_i/2$.
By $\mathcal{E}_M$, there exists a matching $M$ in $G$ of size at least $|A_i|-\mu p_i n$ such that each edge $e$ satisfies $|e \cap A_i| = 1$ and $|e \cap X_i| = 2$.
Now remove all the edges that are incident to a vertex in $T_0$ from $M$ to produce a matching $M'$. By $\mathcal{F}$, $|M| - |M'| \le 10 p_i \mu n$.
Adding the edges of $M'$ to $T_{i-1}$ then yields a partial embedding of $T_i$, where at most $11 p_i \mu n$ edges of $T_i$ are not embedded yet. We now show that one can use vertices of $R$ to embed these edges (which would then show that the first part of~\eqref{algoa} is satisfied).

As every vertex $v$ in $A_i$ is incident to at least $10^4 \mu n $ edges $e$ such that $e \setminus \{v\} \subseteq R$ by $\mathcal E_{0}$, \eqref{eq: TR size} ensures that one can greedily choose edges to complete the embedding of $T_i$ by using at most $2(11 p_i \mu n) \leq 30 p_i \mu n$ vertices in $R$ i.e. $|(V(T_i)\setminus V(T_{i-1})) \cap R| \leq 30 p_i \mu n$, so~\eqref{algoa} is satisfied, a contradiction to the maximality of $i$.

Suppose $i=s+1$. By \eqref{it:path} of Lemma~\ref{lem_tree_split}, in this case, $T_{s+1}$ is obtained by adding at most $\mu n$ vertex-disjoint bare paths of length 3 to $T_{s}$. Conditioning on $\mathcal{E}_0$, for every pair of vertices $u, v \in V(G)$, there are at least $100\mu n $ internally vertex-disjoint $u$--$v$ paths (of length 3) such that all of their internal vertices are in $R$.
By \eqref{eq: TR size}, the embedding of $T_{i-1}$ already used at most $40 \mu n$ vertices from $R$, so for any pair of vertices $u, v \in V(G)$, there are at least $60 \mu n$ internally vertex-disjoint $u$--$v$ paths (of length 3) remaining. As each path of length 3 uses at most 7 vertices in $R$, one can greedily find at most  $\mu n$ vertex-disjoint paths (of length 3) that are required for embedding $T_{s+1}$ in $G$. As we have used at most $7 \mu n$ vertices of $R$, \eqref{algob} is satisfied, a contradiction to the maximality of $i$. 

Hence, we have $i-1 = \ell$, and this completes the proof of \Cref{main_thm}.

\medskip
\noindent{\bf Concluding remarks.}
As noted before, a complete graph on $n$ vertices contains all possible trees on at most $n$ vertices.
This simple fact motivated the tree packing conjecture of Gy\'arf\'as and Lehel~\cite{gyarfas1978packing},
which
states that for $n \in \mathbb{N}$, any given set of trees $T_1, T_2, \ldots, T_n$ with $|V(T_i)|=i$ can be packed into the complete graph $K_n$. 
This notorious conjecture since 1976 has driven a lot of research and it still remains open (see e.g., \cite{allen2021tree, joos2019optimal} for results towards this conjecture). 

As our theorem guarantees that an $n$-vertex Steiner triple system contains all possible hypertrees with at most $(1-o(1))n$ vertices, it is natural to ask if a corresponding ``packing" statement for Steiner triple systems also holds. 
This question was in fact already asked by Frankl, as recorded in~\cite{Elliott2019}, in the following form:
what is the largest integer $s$ such that any $s$ hypertrees $T_3, T_5, T_7, \ldots, T_{2s+1}$ with $|V(T_i)|=i$ can be packed into every $n$-vertex Steiner triple system? Indeed, this question is a natural analogue of the tree packing conjecture, since every hypertree contains an odd number of vertices. Frankl showed that any given set of hypertrees $T_3, T_5, T_7, \ldots, T_{(n+3)/2}$ can be packed into every $n$-vertex Steiner triple system but it is not known if a larger set of hypertrees can be embedded. %We wonder whether our methods can be useful to make progress on this question.

By analyzing our proof more carefully with some modifications, one may prove a minimum-degree version of our theorem. That is, there exists some $\delta>0$ such that the following holds: If $\varepsilon = \Omega(n^{-\delta})$ and $n$ is sufficiently large, then every $n$-vertex linear $3$-graph $G$ with the minimum degree at least $n(\frac{1}{2} - (\frac{\varepsilon}{\log n})^{100})$ contains any hypertree $T$ with at most $(1- \varepsilon )n$ vertices.
By repeatedly applying this to a Steiner triple system and deleting low degree vertices, one can also show that any given set of hypertrees $T_{n - j-t\log^{100} n}$, $j=0,2,4,\cdots,2t$,
with $|V(T_{i})|=i$ pack into every $n$-vertex Steiner triple system for an appropriate choice of $t = \Theta( \frac{n}{\mathrm{polylog}(n)})$. 

\vspace{3mm}

\noindent{\bf Acknowledgement.}
SI is supported by the POSCO Science Fellowship of POSCO TJ Park Foundation, the Institute for Basic Science (IBS-R029-C4), and in part by National Research Foundation of Korea (NRF) grant funded by the Korea government MSIT NRF-2022R1C1C1010300. 
JK is supported by  the POSCO Science Fellowship of POSCO TJ Park Foundation. 
JL is supported by National Research Foundation of Korea (NRF) grant funded by the Korea government MSIT NRF-2022R1C1C1010300, Samsung STF Grant SSTF-BA2201-02, and Institute for Basic Sciences (IBS-R029-C4).
AM is supported by the European Research
Council (ERC) under the European Union's Horizon 2020 research and innovation programme (grant agreement no.\ 786198) and also in part by the EPSRC, grant no.\ EP/S00100X/1 (A.~Methuku).

\begin{appendices}
\section{Hypertree splitting}\label{sec_tree_partition}
The very first step towards our proof of~\Cref{main_thm} is~\Cref{lem_tree_split}, which splits the given hypertree $T$ into ``manageable" pieces.
Our proofs in this section will closely follow that of~\cite{Montgomery2020}, whose first step is the following lemma for $2$-graphs. 
For a (graph) tree $T$, a path $P$ in $T$ is a \emph{bare path} if all the internal vertices of $P$ have degree two. 

\begin{lemma}[\cite{Montgomery2020}]\label{lem_leave_path}
    Let $\ell, m \geq 2$ be integers and let $T$ be a tree with at most $\ell$ leaves.
    Then there exist vertex-disjoint bare paths $P_1, P_2, \cdots, P_s$ of length $m$ such that 
    \begin{align*}
        |V(T-P_1-P_2-\cdots-P_s)| \leq 6m\ell + \frac{2|V(T)|}{m+1},
    \end{align*}
    where $T-P$ denotes the graph obtained by removing all of internal vertices of $P$ from $T$.
\end{lemma}
By considering the breadth-first search (BFS) tree of the line graph of a hypertree, this lemma can be adapted for hypertrees. To that end, we need to specify the corresponding definitions for 3-uniform hypergraphs. 
A \emph{semi-bare path} $P$ in a hypertree $T$ is a path such that 
edges in $T\setminus E(P)$ are only incident to the vertices in its end pairs.
This is a weaker notion than a bare path (in hypertrees); 
in \Cref{fig:bare_path}, green edges plus the right red edge form a semi-bare path but not a bare path.
For a semi-bare path $P$, denote by $T-P$ the hypergraph obtained from $T$ by removing all vertices of $P$ except the vertices in the edges that contains one of the end pairs of $P$.

\begin{lemma}\label{lem_leave_path2}
 Let $\ell, m \geq 2$ be integers. 
 Let $T$ be a hypertree with at most $\ell$ leaf edges.
 Then there exist edge-disjoint semi-bare paths $P_1, P_2, \cdots, P_s$ of length $m+1$ such that
 \begin{align*}
     e(T-P_1-P_2-\cdots-P_s)\leq 6m\ell + \frac{2e(T)}{m+1}.
 \end{align*}
\end{lemma}
\begin{proof}
    If all the edges of $T$ are leaf edges, then $e(T)=\ell$, so the conclusion trivially holds.
    We may thus assume that there is a non-leaf edge $e^* \in E(T)$.
    Let $\mathcal{G}$ be the auxiliary (2-)graph on the edge set $E(T)$ where $e_1$ and $e_2$ are adjacent if they intersect.
    
    Let $T_\mathcal{G}$ be the BFS tree of $\mathcal G$ rooted at $e^*$.
    Then a leaf of $T_\mathcal{G}$ corresponds to a leaf edge of $T$.
    Indeed, an edge $e$ of $T$ is a leaf edge if and only if it has only one vertex incident to other edges. 
    Thus, if there exists a parent $e'$ of $e$ in the rooted tree $T_\mathcal{G}$, then 
    the other edges that intersect with $e$ also intersect with~$e'$ and the BFS puts them in the same depth as $e$. This makes $e$ a leaf vertex in $T_\mathcal{G}$.
    Conversely, if $e$ is not a leaf edge in $T$, then deleting $e$ disconnects the hypertree, which also disconnets the BFS tree $T_\mathcal{G}$. In particular, $e$ is not a leaf vertex in $T_\mathcal{G}$.
    Since the root $e^*$ of $T_\mathcal{G}$ is not a leaf edge in $T$, an edge $e\in E(T)$ is a leaf edge if and only if it is a leaf (vertex) in $T_\mathcal{G}$.
    Therefore, $T_\mathcal{G}$ has at most $\ell$ leaves.

    Moreover, a bare path of length $m$ in $T_\mathcal{G}$ corresponds to a semi-bare path of length $m+1$ in~$T$. 
    Indeed, the $m-1$ internal vertices of a bare path of length $m$ in $T_\mathcal{G}$ form edges of a bare path of~$T$ because of the BFS, which uniquely extends to a semi-bare path of length $m+1$. 
    Conversely, the hyperedges of a semi-bare path in $T$ becomes the vertex set of a bare path.
    Therefore, applying Lemma~\ref{lem_leave_path} to $T_\mathcal{G}$ yields the desired result.
\end{proof}

We are now ready to prove Lemma~\ref{lem_tree_split}.

\begin{proof}[Proof of Lemma~\ref{lem_tree_split}]
We construct a decreasing sequence $T=T'_0 \supseteq T_1'\supseteq \cdots \supseteq T_\ell'$, which will yield the desired increasing sequence $T_k = T_{\ell-k}'$.
 Let $m:=\lceil 10^3/\mu \rceil$. Starting from $T$, we iteratively remove a matching leaf set of size at least $\mu n / (50 m D)$ as many times as possible.
 There are at most $50mD/\mu$ such iterations, so starting from $T_0'=T$ gives the resulting tree $T_{k}'$, $k\leq 50mD/\mu$.

We say that a non-leaf vertex $v$ in a leaf edge is the \emph{parent} of the edge or simply a parent if it is a parent of a leaf edge.
 By the choice of $T_{k}'$, there are at most $\mu n / (50 m D)$ parent vertices. For every vertex $v$ which is a parent of at least $D$ leaf edges, remove all the leaf edges incident to $v$ to obtain a smaller hypertree $S$. 
 Then $S$ contains at most $\mu n /50 m$ leaf edges, since each leaf edge in $S$ is either a leaf edge in $T_k'$ which shares its parent with at most $D$ leaf edges or it contains a parent in $T_k'$.

 Then by \Cref{lem_leave_path2}, $S$ contains edge-disjoint semi-bare paths $P_1, \cdots, P_r$ of length $m+1$ such that $e(S-P_1-\cdots-P_r) \leq 3\mu n /25  + 2e(S)/(m+1) \leq \mu n/2$ since $S$ has at most $\mu n /50 m$ leaf edges.
 As there are at most $\mu n /(50mD)$ vertices in $S$ whose degrees differ from theirs in $T_k'$,
 at least $\max\{0, r-\mu n / (50 m D)\}$ of the semi-bare paths $P_1,P_2,\cdots,P_r$ are still semi-bare paths of $T_{k}'$.
 Let $Q_1, \ldots, Q_{r'}$ be such semi-bare paths with $r'\geq \max\{0, r-\mu n/(50mD)\}$.
 As each of the $Q_i$'s contains a bare path of length three in the middle, we can remove the vertices in these bare paths to obtain $T_{k+1}'$. Indeed, as $Q_i$'s are semi-bare paths that are edge-disjoint, they are vertex-disjoint except the vertices in their end pairs, thus the removed bare paths are vertex-disjoint.

 Starting from $T_{k+1}'$, at each step we delete one leaf edge from the remaining edges in each $Q_i$ except the edges at the end of $Q_i$. This yields $T_{k+2}',T_{k+3}',\cdots,T_{k+m-3}'$, where each $T_{k+t+1}'$ is obtained by deleting vertex-disjoint leaf edges from $T_{k+t}'$.

The final step from $T_{k+m-3}'$ to $T_{k+m-2}'$ essentially repeats what we did to obtain $S$.
 For every vertex that is a parent of at least $D$ leaf edges in $T'_{k+m-3}$, remove all the leaf edges incident to $v$ to obtain $T_{k+m-2}'$ and take $\ell=k+m-2$.
 Then $\ell \leq m + 50 m D / \mu \leq 10^5 D \mu^{-2}$.
 
 It remains to prove that $e(T_{k+m-2}') \leq \mu n$.
 First, recall that $e(S-P_1-\cdots-P_r) \leq \mu n/{2}$.
 Also, if a vertex is a common parent of $D$ leaf edges in $T_{k}'$, then it is so in $T_{k+m-3}'$ as well. Thus, if an edge in $T_{k+m-3}'$ is not removed when obtaining $T_{k+m-2}'$, then it must have remained when obtaining $S$ from $T_{k}'$, i.e., $E(T_{k+m-2}') \subseteq E(S)$.
 Thus, $T_{k+m-2}'$ is contained in both $S$ and $T_{k+m-3}'=T_{k}'-Q_1-Q_2-\cdots -Q_{r'}$,
 which implies that
 \begin{align*}
     e(T_{k+m-2}') \leq e(S-Q_1-\cdots-Q_{r'}) \leq e(S-P_1-\cdots-P_r) + m \cdot \mu n / (50 m D) \leq \mu n. \tag*{\qedhere}
 \end{align*}
 This completes the proof of the lemma.
\end{proof}

\end{appendices}

\end{document}